\documentclass[11pt,centertags,oneside]{amsart}
\usepackage{amsmath,amstext,amsthm,amscd,typearea,hyperref}
\usepackage{amssymb}
\usepackage{a4wide}
\usepackage[mathscr]{eucal}
\usepackage{typearea}
\usepackage{charter}
\usepackage[a4paper,width=16.5cm,top=3cm,bottom=3cm]{geometry}

\numberwithin{equation}{section}

\allowdisplaybreaks
\emergencystretch=\maxdimen
 \usepackage{cleveref}

\usepackage{multicol}

\usepackage{xcolor}

\newtheorem{theorem}{Theorem}[section]
\newtheorem{definition}[theorem]{Definition}
\newtheorem{proposition}[theorem]{Proposition}
\newtheorem{corollary}[theorem]{Corollary}
\newtheorem{lemma}[theorem]{Lemma}

\newtheorem{thm}{Theorem}

\theoremstyle{definition}
\newtheorem{remark}[theorem]{Remark}

\newcommand{\ZZ}{{\mathbb{Z}}}
\newcommand{\CC}{{\mathbb{C}}}

\newcommand{\RR}{{\mathbb{R}}}
\newcommand{\PP}{{\mathbf{P}}}

\newcommand{\FF}{{\mathbb{F}}}

\DeclareMathOperator{\rk}{rank}

\DeclareMathOperator{\Sym}{Sym}

\DeclareMathOperator{\ord}{ord}

\newcommand{\A}{\mathcal{A}}
\newcommand{\TP}{\mathcal{T}}
\newcommand{\E}{\mathcal{E}}

\newcommand{\GL}{\mathrm{GL}}
\newcommand{\SL}{\mathrm{SL}}

\newcommand{\PSL}{\mathrm{PSL}}

\newcommand{\SU}{\mathrm{{SU}}}
\newcommand{\SO}{{\mathrm{SO}}}

\newcommand{\Id}{\mathrm{Id}}
\newcommand{\e}{\mathrm{e}}

\newcommand{\J}{\mathbf{j}}

\newcommand{\T}{{\mathsf{T}}}

\newcommand{\rank}{\mathrm{rank}}

\newcommand{\tr}{\mathrm{tr}}

\newcommand{\M}{\mathbf{m}}

\newcommand{\1}{\mathbf{1}}

\providecommand{\RR}{\mathbb{R}} \providecommand{\ZZ}{\mathbb{Z}}

\definecolor{han}{rgb}{1.0, 0, 0}

\newcommand{\pr}[1]{\mathbf{P}  \left[ #1 \right]}

\newcommand{\ex}[1]{\mathbf{E}[#1]}
\newcommand{\Gr}{\mathscr{G}}
\renewcommand{\S}{\mathsf{S}}
\newcommand{\I}{\mathsf{I}}

\renewcommand{\d}{\mathrm{dist}}

\usepackage{fancyhdr}

\title{Optimal linear sofic approximations of countable groups}

\author{Keivan Mallahi-Karai}
\address{Constructor University, Bremen, Germany}
\email{kmallahika@constructor.university}

\author{ Maryam Mohammadi Yekta}
\address{University of Waterloo, Ontario, Canada}
\email{m5mohamm@uwaterloo.ca} 
 
\begin{document}

\begin{abstract}
Let $G$ be a group. The notion of linear sofic approximations of $G$ over an arbitrary field $F$ was introduced and systematically studied by Arzhantseva and P\u{a}unescu \cite{AP17}. Inspired by one of the results of  \cite{AP17}, we introduce and study the invariant $\kappa_F(G)$ that captures the quality of linear sofic approximations of $G$ over $F$. In this work we show that when $F$ has characteristic zero and $G$ is linear sofic over $F$, then $\kappa_F(G)$ takes values in the interval $[1/2,1]$ and $1/2$ cannot be replaced by any larger value. Further, we show that under the same conditions, $\kappa_F(G)=1$ when $G$ is torsion free. These results answer a question posed by Arzhantseva and P\u{a}unescu \cite{AP17} for fields of characteristic zero. 
One of the new ingredients of our proofs is an effective non-concentration estimates for random walks on finitely generated abelian groups, which may be of independent interest. 
\end{abstract}

\clearpage\maketitle
 

\medskip


\setcounter{tocdepth}{1}


\tableofcontents

\section{Introduction}
Let $\Gr= \{ (G_n, \d_n) \}_{n \ge 1 } $ be a family of groups, each equipped with a bi-invariant bounded metric. Bi-invariance means that for every $x,y, g_1, g_2 \in G_n$, we have the  
equality $$ \d_n( g_1xg_2, g_1yg_2)= \d_n(x, y)$$ 
A $\Gr$-approximation of a countable group $G$ consists of an increasing sequence $(n_k)_{k \ge 1}$ of positive integers and
a sequence $(\phi_k)_{k \ge 1}$ of maps $$\phi_k:  G \to G_{n_k}, \quad k \ge 1$$ satisfying the following two properties:

\begin{enumerate}
\item (Asymptotic homomorphism) For all $g, h \in G$, one has $$ \lim_{k \to \infty} \d_{n_k} ( \phi_k( gh), \phi_k(g) \phi_k(h) )=0.$$
\item (Uniform injectivity) There exists $\kappa >0$ such that for all $g \in G \setminus \{ e_\Gamma \}$, 
$$ \limsup_{k \to \infty} \d_{n_k} ( \phi_k ( g), e_{_{G_{n_k} }}) \ge \kappa.$$
\end{enumerate}

We will  then also say that $G$ is $\kappa$-approximable by  $\Gr$.
Perhaps the most prominent  and well-studied classes of approximable groups  are the sofic and hyperlinear groups, which correspond, respectively, to approximation by the family of symmetric groups equipped with the normalized Hamming  distance and the family  of unitary groups  equipped with  normalized Hilbert-Schmidt distance. 
The class of sofic groups was introduced by Gromov in connection with the so-called Gottschalk surjunctivity conjecture  \cite{Gromov99}, while the terminology is due to Weiss \cite{Weiss00}. Hyperlinear groups first appeared in the context of Conne's embedding conjecture. The term hyperlinear was coined by R\u{a}dulescu \cite{Rad08}. 
 Sofic groups are shown to be hyperlinear \cite[Theorem 2]{ElekSzabo},  It is unknown whether every group is sofic, or even hyperlinear. 

The class of linear sofic groups over an arbitrary field was introduced by Arzhantseva and P\u{a}unescu \cite[Definition 4.1 and the paragraph following Definition 4.2]{AP17} who proved fundamental results about this class of groups. 
This mode of approximation defining linear sofic groups uses general linear groups (over a general field $F$ fixed in the discussion) as target groups while the metric is defined using the normalized rank. In this regards, linear sofic groups provide a hybrid form of approximation 

 In order to define this metric, let $F$ be a field.  For $d \times d$ matrices $A, B \in \GL_d(F)$, we define
\[ \rho_d(A, B):= \frac{1}{d} \rank(A-B). \]

The following definition of linear sofic groups will be more convenient for our purpose. The equivalence of two definitions is proven in 
\cite[Proposition 4.4]{AP17}. 

\begin{definition}\label{linsof} Let $F$ be a field, $G$ a countable group and $0 < \kappa \le 1$. We say that $G$ is $\kappa$-linear sofic over $F$ if for every finite set $S \subseteq G$ and every $ \delta>0$ and every $0 \le  \kappa' < \kappa$, there exists
$d \ge 1$ and a map $\phi: S \to \GL_d(F) $  
satisfying the following two properties: 

\begin{enumerate}
\item[(AH)] For all $g, h , gh \in S $, one has $ \rho_d ( \phi( gh), \phi(g) \phi(h) )<  \delta$. 
\item[(D)] For all $g \in S \setminus \{ e \}$, $\rho_d ( \phi ( g), \Id ) \ge \kappa'$. 
\end{enumerate}
\end{definition}

Such a map is called an $(S, \delta, \kappa')$-map. Roughly speaking, (AH) guarantees that $\phi$ is {\it almost} a homomorphism, while
(D) shows that distinct elements are separated out. 
Following \cite{AP17} we say that $G$ is linear sofic over $F$ if it is $\kappa$-linear sofic for some $\kappa>0$. It is clear that if $\kappa_1<\kappa_2$, then every $\kappa_2$-linear sofic group is $\kappa_1$-linear sofic. For a countable group $G$, we write 
\[ \kappa_F(G)= \sup \{ \kappa \ge 0:  G \text{ is } \kappa \text{-linear sofic } \text{over } F  \}  \in (0,1]. \]
Note that whenever $G$ is not $\kappa$-linear sofic over $F$ for any $ \kappa>0$, we define $\kappa_F(G)$ to be zero.

\begin{remark}
The notion of metric approximation can also be defined using the notion of metric ultraproducts. This alternative definition allows one to avoid limiting processes that require passing to subsequences, and thereby simplifies certain arguments, see \cite{AP17} for examples.  Since this point of view will not provide us with any special advantages, we will not use this definition. 
\end{remark}

\subsection{The amplification argument}
Let $\Gr= \{ (G_n, \d_n) \}_{n \ge 1 } $ be as above, and assume that the diameter of $G_n$ with respect to $\d_n$ is normalized to be $1$. It is natural to ask whether a $\kappa$-approximable group for some $\kappa>0$ is always $1$-approximable. 
Elek and Szab\'{o} \cite{ElekSzabo} proved that this is the case for sofic groups. A similar statement (with a modified proof) holds for 
hyperlinear groups. Note that this implies that analogously defined $\kappa_{\textrm{sofic} }$ and $\kappa_{\textrm{hyperlinear} }$ can only take values in the
set $\{ 0, 1 \}$, and the longstanding open question asking whether all groups are sofic is equivalent to $\kappa_{\textrm{sofic}} (G)=1$ for all groups $G$. 

Let us recall that both proofs are based on a basic tool, often referred to as {\it amplification}, which uses the
identity
\begin{equation}\label{trace}
 \tr(a \otimes b)= \tr (a) \tr(b).
\end{equation}
for matrices $a$ and $b$. 
This identity allows one to show that there exists a function $f: (0,1) \to (0,1)$ such that if one starts with a map $\phi: S \to \S_n$ with $\d_{\textrm{Hamm}} ( \phi ( g), e ) \ge \beta$ then the tensor power $\phi^{\otimes 2}: S \to \S_{n \times n}$ 
defined by $\phi^{\otimes 2}(g)(i,j)= (\phi(g)(i), \phi(g)j)$ 
satisfies $\d_{\textrm{Hamm}} ( \phi^{\otimes 2}(g), e ) \ge f(\beta)$.
Moreover, starting from any $\beta$, the sequence of iterates $f^{(n)}(\beta)$ converges to $1$. Hence, by iterating the tensor power operation, one can arrive at arbitrarily well sofic approximation. The case of hyperlinear groups is dealt with in a similar fashion. 
As it was observed in \cite{AP17},  \eqref{trace} does not have an analog for linear sofic approximations. In  \cite{AP17}, Arzhantseva and P\u{a}unescu invented a new amplification argument to prove that 
every linear sofic group is $1/4$-linear sofic. A particularly innovative aspect of this argument is that it tracks two different quantities that when coupled together can be used to control the distance to the identity. Then using clever properties of ranks of tensor powers they prove that this amplification argument works.  The question of whether the constant 
$1/4$ can be improved is left open in \cite{AP17}. We will build upon their work to answer this question in the case of fields of characteristic zero.

\subsection{Statement of results}
In this paper, we will address the question of optimality of linear sofic approximations. The main results of this paper is the  theorem below.

\vspace{2mm}

\begin{thm}\label{main-theorem}
Let $G$ be a countable linear sofic group over $\CC$. Then 
\begin{enumerate}
\item If $G$ is torsion-free, then $G$ is $1$-linear sofic over $\CC$. 
\item Unconditionally, $G$ is $1/2$-linear sofic over $\CC$. Moreover, the constant $1/2$ cannot be improved. 
\end{enumerate}
\end{thm}

Note that the assertion in Theorem \ref{main-theorem} is in stark contrast with the case of sofic and hyperlinear groups.  An interesting observation in \cite{AP17} (see the paragraph before Proposition 5.12) is that the amplification argument does not {\it see} the interaction between group elements and will equally work for a subset of a group. Theorem  \ref{main-theorem}, however, shows that the optimal constant does indeed depend on the group structure and can even change by passing to a subgroup of finite index.

\begin{remark}
\textcolor{black}{Although we stated Theorem  \ref{main-theorem} over $\CC$, one can easily see that it implies 
the same statement over all fields of characteristic zero; see Remark \ref{why}. However, an important part of the argument that is based on Lemma \ref{lemma:jordan} does not work when $F$ has positive characteristic. See Remark \ref{remarkbad} for more details. Henceforth, we write $\kappa$ instead of $\kappa_\CC$. }
\end{remark}

We briefly outline the proof of Theorem \ref{main-theorem}, which is following the main strategy of \cite{AP17}. Given a finite
subset $S \subseteq G$ and $ \delta_0 >0$, we start with an $(S, \delta_0, 0.24)$-map $\phi_0$ (in the sense of Definition \ref{linsof}) provided by \cite{AP17}. Using a sequence of functorial operations (see \ref{functor} for definitions) we replace $\phi_0$ with an $(S, \delta_0, 0.23)$-map which has the additional property that for
every $g \in S$, at least $1/100$ of eigenvalues of matrix $\phi(g)$ are $1$. We will then show that the rank of tensor powers of $\phi(g)$ are controlled by the return probability of a certain random walk on a finitely generated subgroup
of $\CC^{\ast}$. We will establish required {\it effective} non-concentration estimates in Theorem \ref{return}. This part of proof uses a variety of tools ranging from Fourier analysis to additive combinatorics. Let us note that the effectiveness of these bounds is a key element of the proof: as the asymptotic homomorphism condition (AH) deteriorates after every iteration of tensor power, we need to know in advance the number of required iterations so that we can start with an appropriate $ \delta_0$. 
The counter-intuitive move of adding ones as eigenvalues is  needed for this purpose. When  $\phi_0(g)$ is close to unipotent, this argument completely breaks down. In this case, again using the method of  \cite{AP17} we will instead show 
the normalized number of Jordan block of tensor powers tends to zero  with an effective bound for speed.  This is carried out in Theorem \ref{jdec} by translating the problem to estimating integrals of certain trigonometric sums. In summary, our proof can be viewed as a version of amplification argument where we use additional functors in the process.

Another result of this paper involves determining $\kappa(G)$ for finite groups. 

\begin{thm}\label{main-theorem2}
Let $G$ be a finite group. 
\begin{enumerate}
\item There exists a finite-dimensional linear representation 
$\psi: G \to \GL_d(\CC)$ of $G$ such that 
\[ \kappa(G)=  \min_{g \in G} \rho_d( \psi(g), \Id). \]
In particular, $\kappa(G)$ is a rational number. 
\item  
$\kappa(G)=1$ iff $G$ has a fixed-point free complex representation.
\item 
Let $\ZZ_p$ denote the cyclic group of order $p$. For prime $p$ and $n \ge 2$ we have $$\kappa(\ZZ_p^n)= \frac{p^n - p^{n-1}}{p^n-1}. $$
In particular, $\kappa(\ZZ_2^n) \to 1/2$ as $n \to \infty$. 
 \end{enumerate}
\end{thm}

One of the main ingredients in the proof of Theorem \ref{main-theorem2} is the notion of stability. 
Broadly described, stability of a group in a mode of metric approximation demands that almost 
representations be small deformations of exact representations. Finite groups are easily seen to 
enjoy this property with respect to linear sofic approximations, see Proposition~\ref{stable}. Once this is established, the problem reduces to representation theory of finite groups. Proof of Part (a) of Theorem 
\ref{main-theorem2} is based on the simplex method in linear programming. Part (3) implements this for 
groups $ \ZZ_p^n$, $n \ge 2$. We finally remark that 
all finite groups to which (2) of Theorem \ref{main-theorem} applies have been classified by
Joseph Wolf, \cite{Wolf}. They include groups such as $\PSL_2(\FF_5)$. The next result establishes the value of $\kappa_F(G)$ 
for certain classes of groups over fields of positive characteristic. 

%

\begin{thm}\label{main-theorem3}
Let $F$ be a field of characteristic $p$,  and let $G$ be a finite group such that $p$ is the smallest prime dividing $|G|$. Then 
\[ \kappa_F( G) = 1- \frac{1}{p}. \]
\end{thm}

One of the problems posed in \cite{AP17} is whether the notions of linear sofic approximation over $\CC$ and $\FF_p$ are equivalent.  Theorem \ref{main-theorem2} and Theorem \ref{main-theorem3} together show that in general the values of
$\kappa_\CC(G)$ and $\kappa_{\FF_p}(G)$ need not coincide for a finite group $G$. This may be viewed as a quantitative reason for the difficulty of the problem of equivalence. We note that quantitative approaches to other metric approximations have also been considered before, see \cite{AC20}.

This paper is organized as follows: 
In Section \ref{prelim}, we will collect basic facts related to the rank metric, and basic theory of
random walks on abelian groups and explain how they relate to our question. 
In Section \ref{prob-lemmas}, we prove various non-concentration estimates for random walks on abelian groups. 
Section \ref{sec:jordan} contains the proof of Theorem \ref{jdec}, involving matrices in Jordan canonical form. 
These ingredients are put together in Section \ref{proof} to prove Theorem \ref{main-theorem}, except for the optimality claim.  
The optimality, as well as the proof of Theorems \ref{main-theorem2} and \ref{main-theorem3}, are discussed in  Section \ref{sec:finite}. 

\vspace{2mm}

{\it Acknowledgement} The authors would like to thank Goulnara Arzhantseva for helpful comments and suggestions on an earlier version of this paper. We also thank the anonymous referee for a careful reading of the paper and numerous remarks that significantly improved the exposition of this paper. Special thanks are due to Iosif Pinelis for providing the reference to Theorem \ref{thm:essen}.

\section{Preliminaries and notation} \label{prelim}
In this section, we will set some notation and gather a number of basic facts needed in the rest of the paper. 
We will denote the group of invertible $d \times d$ matrices over the field $\CC$ by $\GL_d(\CC)$. This space can be turned into a metric space by defining for $A, B \in \GL_d(\CC)$:
\[ \rho_d( A, B)= \frac{\rank (A-B)}{d}. \]
We will often suppress the subscript $d$ and simply write $\rho(A, B)$. Every 
$A \in \GL_d(\CC)$ has $d$ eigenvalues $ \lambda_1, \dots, \lambda_d$, each counted with multiplicity. We write
\[ \M_1(A)= \frac{1}{d} \# \{ 1 \le i \le d : \lambda_i =1 \}. \]
For $j \ge 1$, let us denote by $W_j$ the set of all $j$-th roots of unity in $\CC$. It will later be convenient to consider the following quantity:
\[ \M^{\le r}(A)=\frac{1}{d} \# \{ 1 \le i \le d : \lambda_i  \in \bigcup_{  j=1}^{  r}  W_j \}. \]
The number of Jordan blocks of $A$ (in its Jordan canonical form) will be denoted by $\J(A)$. The number of Jordan blocks corresponding with $1$ on the diagonal will be denoted by $\J_1(A)$. Note that 
\[ \J_1(A) = \dim \ker (A - \Id) = ( 1- \rho(A, \Id) )d. \]

The following lemma plays a key role in the arguments used in \cite{AP17}:

\begin{lemma}[\cite{AP17}]\label{basiclemma}
For every $A \in \GL_d(\CC)$, we have 
\[ \rho(A, \Id) \ge \max( 1- \M_1(A), 1- \J(A)). \]
\end{lemma}

The next lemma will be essential for tracking the multiplicity of eigenvalue $1$ in tensor powers of a matrix $A$: 
  
\begin{lemma}[\cite{AP17}, Lemma 5.1]\label{tensor}
Suppose that $\{  \lambda_1, \dots, \lambda_d \} $ is the set of eigenvalues of a $d \times d$ complex matrix $A$, each counted with multiplicity. Then, for $k \ge 1$ the set of eigenvalues of 
$A^{\otimes k}$ counted with multiplicity is given by
\[ \{ \lambda_{i_1} \cdots \lambda_{i_k}: 1 \le i_j \le d, \quad 1 \le j \le k \}. \]
\end{lemma}

\begin{proof}
There exists $P \in \GL_d(\CC)$ such that $P^{-1}AP$ is upper-triangular. Hence, without loss of generality, we can assume that $A$ is upper-triangular with diagonal entries $ \lambda_1, \dots, \lambda_d$. It is easy to see that $A^{\otimes k}$ will be upper-triangular in an appropriate ordering of the bases, with diagonal entries given by the list above.  Special case $k=2$ is dealt with in  Lemma 5.1 of \cite{AP17}. 
\end{proof}

\subsection{Three functorial operations}\label{functor}
Let us now consider three functorial operations that can be applied to a family of matrices in $\GL_d(\CC)$. These will be used to replace an $(S, \delta, \kappa)$-map by an $(S, \delta', \kappa')$-map, which has some better properties. An alternative point of view is that each operation can be viewed as post-composition of the initial map by a representation of $\GL_n$.

\begin{enumerate}
\item (Tensors) Consider the representations $$\Psi_{T,m}: \GL_d(\CC) \to \GL_{d^m}(\CC), \quad A \mapsto A \otimes \cdots \otimes A,$$ 
where $m$ denotes the number of tensors. We will denote $\Psi_{T,m}(A)$ by $\T^m A$. 

\item (Direct sums) For $ m \ge 1$, let 
$$\Psi_{S,m}: \GL_d(\CC) \to \GL_{md}(\CC), \quad A \mapsto A \oplus \cdots \oplus A,$$
where the number of summands is $m$.  Instead of $\Psi_{S,m}(A)$, we write $\S^m A$.

\item (Adding Identity)
For $m \ge 0$, consider the representation 
$$ \Psi_{I, m} : \GL_d(\CC) \to \GL_{m+d}(\CC), \quad A \mapsto A \oplus \Id_m,$$ 
where $\Id_m$ is the identity matrix of size $m$. We write $\I^m A$ for $ \Psi_{I, m} (A)$.

\end{enumerate}

\begin{lemma}\label{100}
Let $A$ and $B$  be   $d \times d$ matrices. Then for $m,n \ge 1$ we have 
\begin{enumerate}
\item $\M_1( \I^{md} \S^n A)=  \M_1(A)+ \frac{m}{n}.$
\item $\rho(  \I^{md} \S^n A,  \I^{md} \S^n B) \le \rho(A, B). $
\item $\rho( \I^{md} \S^n A, \Id)  = \frac{n}{n+m} \rho(A, \Id).$
\end{enumerate}
\end{lemma}

\begin{proof}
Parts (1) and (2) are straightforward computations. For (3) note that
\[ \dim (\ker  (\I^{md} \S^n A - \I^{md} \S^n \Id) )= n \dim \ker ( A- \Id) + md. \] 
The claim follows by a simple computation. 
\end{proof}

\begin{proposition}\label{addone}
Let $G$ be a linear sofic group. Then for every finite $S \subseteq G$, $ \delta>0$ and $\kappa < 0.24$
there exists an 
$( S, \delta, \kappa)$-map $\phi$ such that $\M_1( \phi(g) ) \ge 0.01$ for all $g \in S \setminus \{ e \}$. 
\end{proposition}

\begin{proof}
We know from \cite{AP17} that $\kappa(G) \ge 1/4$. 
Since $ \kappa + 0.01 < 0.25 \le \kappa(G)$, there exists a $( S, \epsilon, 0.01+\kappa)$-map, which we denote by $\phi_0$. Set $\phi=\I^{d} \S^{100} \phi_0$.  By Lemma \ref{100}, we 
have $  \M_1( \phi(g)) \ge 0.01$. Moreover, for every $g \in S \setminus \{ e \}$ we have
\[ \rho( \phi(g), \Id) \ge  \frac{100}{101} (\kappa + 0.01)\ge \kappa. \] 
\end{proof}

\begin{lemma}\label{bad} 
Let $A$ and $B$  be   $d \times d$ matrices. 
Then $\rho( \T^n A, \T^n B) \le n \rho(A, B)$. 
\end{lemma} 

\begin{proof}
For $n=2$, we have
$A \otimes A - B \otimes B= A \otimes (A- B)+ ( A  -B) \otimes B.$ The claim follows from \cite[Proposition 5]{AP17}
and the triangle inequality. The general case follows by a simple inductive argument. 
\end{proof}

\subsection{Preliminaries from probability theory}
It will be convenient to reinterpret Lemma \ref{tensor} in a probabilistic language. 
Let $(\Lambda,+)$ be a countable abelian group with the neutral element denoted by $0$.  By a probability measure on $\Lambda$, we mean a map $\mu: \Lambda \to [0,1]$ such that $ \sum_{ a \in \Lambda}\mu(a)=1$. We will then write $\mu= \sum_{ a \in \Lambda } \mu(a) \delta_a$. For $B \subseteq \Lambda$, we define
$\mu(B)= \sum_{ a \in B} \mu(a)$. We say that $\mu$ is finitely supported if there exists a finite set $B \subseteq \Lambda$ such that $\mu(B)=1$. The smallest set with this property is called the support of $\mu$. Probability measures considered in this paper are finitely supported.

The convolution of probability measures measures $\mu_1$ and $\mu_2$ on $\Lambda$ is the probability measure defined by 
\[ (\mu_1 \ast \mu_2) (a)= \sum_{ a_1+ a_2=a} \mu_1(a_1) \mu_2(a_2). \]
One can see that the convolution is commutative and associative. The $k$-th convolution power of $\mu$ will be denoted by $\mu^{(k)}$. Given a probability measure $\mu$ on the group $A$, the $\mu$-random walk on $\Lambda$ (or the random walk governed by $\mu$) is the random process 
defined as follows. Let $(X_k)_{k \ge 1}$ be a sequence of independent random variables, where the law of $X_i$ is $\mu$. Define the process $(S_k)_{k \ge 0}$ by 
$S_0=0$ and $S_k= X_1+ \cdots + X_k$. It is easy to see that the law of $X_k$ is $\mu^{(k)}$. We will use the notation 
$\pr{E}$ to denote the probability of an event $E$. Similarly, $\ex{X}$ denotes the expected value of a random variable $X$. 
 
Given $A \in \GL_d(\CC)$ with eigenvalues $ \lambda_1, \dots, \lambda_d$, define the probability measure
on $\CC$
 \[ \xi_A := \frac{1}{d} \sum_{i=1}^d \delta_{ \lambda_i}. \]
 
\begin{proposition}\label{connection} Let $A$ be a $d \times d$ invertible complex matrix. Then 
\begin{enumerate}
\item $ \xi_A$ is a finitely supported probability measure on the multiplicative group $ \CC^{\ast}:= \CC \setminus \{ 0 \}$. 
\item $\xi_A( 1) = \M_1(A). $
\item For all integer $k \ge 1$ we have
\[ \xi_{ \T^k(A) } = (\xi_A)^{ ( k)}. \]
\end{enumerate}
\end{proposition}

\begin{proof}
Parts (1) and (2) follow from the definition of $\xi_A$.
Part (3)  is an immediate corollary of Lemma \ref{tensor}.
 \end{proof}
 
\section{Effective non-concentration bounds on abelian groups}\label{prob-lemmas}
This section is devoted to proving effective non-concentration bounds for random walks on abelian groups. We will use the additive notation in 
most of this section. However, we will eventually apply these results to subgroups of the multiplicative group of non-zero complex numbers.
 Let $\nu$ be a 
finitely supported probability measure on a finitely generated abelian group $\Lambda$. Our goal is to prove effective upper bounds for the return probability $\nu^{(n)} (0)$.

\begin{lemma}\label{lem:nonconcentration1}
Let $\nu$ be a finitely supported probability measure on $\ZZ^d$ such that $\beta \le \nu(0) \le 1- \beta$ for some
$0<\beta<1/2$. Then 
\[ \nu^{(n)} (0) \le \frac{C}{ \sqrt{ \beta n} } \]
where $C$ is an absolute constant. 
\end{lemma}
 
Note that the key aspect of Lemma \ref{lem:nonconcentration1} is that the decay rate is controlled only by $\beta$ without no further assumption on the distribution of $\nu$. This fact will be crucial in our application. 
The proof of Lemma \ref{lem:nonconcentration1} is based on a non-concentration estimate in classical probability theory. Before stating the theorem, we need a few definitions. The concentration function $Q(X, \lambda)$ of a random variable $X$ 
is defined by 
\[ Q(X, \lambda)= \sup_{ x \in \RR} \PP [ x \le X \le x+ \lambda], \quad \lambda \ge 0. \] 
We will use a theorem of Rogozin \cite{Rogozin61}, which generalizes a special case due to Kolmogorov\cite{Kolmogorov58}. Our statement of the theorem is taken from \cite[Theorem 1]{Esseen65}, where 
a new proof using Fourier analysis is given. A variation of this proof can also be found in \cite[Chapter]{Petrov75}. 

\begin{theorem}[Rogozin]\label{thm:essen}
Suppose $X_1, \dots, X_n$ are independent random variables and $S_n= X_1+ \cdots +X_n$. Then for every non-negative $ \lambda_1,
\dots, \lambda_n \le L$ we have 
we have
\[ Q(S_n, L) \le C \, L  \left(  \sum_{ k =1}^{n} \lambda_k^2 (1- Q( X_k, \lambda_k) ) \right)^{-1/2}, \]
where $C$ is an absolute constant. 
\end{theorem}

\begin{proof}[Proof of Lemma \ref{lem:nonconcentration1}] Let $\iota: \ZZ^d \to \RR$ denote an embedding of $\ZZ^d$ into $\RR$ as an abelian group. Let $\nu'$ denote the push-forward of $\nu$ which is a finitely supported probability measure defined by 
$\nu'(x) = \nu( \iota^{-1} (x) )$ for every $x \in \RR$. Let $X_1, \dots, X_n$ be independent identically distributed random variables with distribution $\nu'$. Then by 
choosing all $ \lambda_i$ equal to $ \lambda>0$ we obtain
\[ \pr{S_n =0} \le  C  ( n(1-Q(X_1, \lambda) )^{-1/2}. \] 
Letting $ \lambda \to 0$ we obtain $\nu'^{(n)}(0)= \pr{S_n =0} \le  C  ( n(1-q ))^{-1/2}$ where
$q= \max_{x \in \RR} \pr{X_i=x}$. Since $\nu'(0)=\nu(0) \ge \beta$, we have $\nu'(x) \le 1- \beta$ for all $x \in \RR \setminus  \{ 0 \}$. 
Since $\nu(0) \le 1-\beta$, we have $q \le 1- \beta$. The claim follows by noticing that $\nu^{(n)}(0_{\ZZ_d})= \nu'^{(n)} (0_{\RR})$. 
\end{proof} 
 
\begin{remark}
It might be tempting to expect a similar upper bound for $\nu^{(n)} (0)$  under the weaker assumption that 
$ \nu(0) \le 1- \beta$. This is, however, not true. To see this, consider the probability measure $\nu_k$ with
$\nu_k(1)= 1- \frac{1}{k}$ and $\nu_k(-k)= \frac{1}{k}$. Although $\nu_k(0)=0$, we have for all $k \ge 1$ 
$\nu_k^{(k+1)}(0)= \frac{k+1}{k} ( 1- \frac{1}{k})^k \approx 1/e$.  
\end{remark}

We will now consider a variant of Lemma \ref{lem:nonconcentration1} for finite cyclic groups $\ZZ_N$. In this case the uniform measure is the stationary measure and hence $\nu^n(0)$ (under irreducibility assumptions) will converge to $1/N$ as $n \to \infty$. Note, however, that {the random walk and its frequency of visits to zero can be bounded by how much the measure is supported on small subgroups.}
 The next lemma provides an effective upper \textcolor{black}{upper} bound for $\nu^{(n)} (0)$ only depending on the mass given to elements of small order.

\begin{lemma}[Non-concentration for random walks on cyclic groups]\label{lem:nonconcentration2}
Let $\nu$ be a probability measure on $\ZZ_N$. Further, suppose that for some $0< \beta  \le 1/2$ and positive integer $r>1$ the following hold:
\begin{enumerate}
\item $\beta \le \nu(0) \le 1- \beta$. 
\item For every subgroup $H \le \ZZ_N$ with $|H| \le r$, we have $ \nu(H) \le 1- \beta$. 
\end{enumerate}
Then for all $n \ge 1$ we have
\[     \nu^{(n)} (0)   \le  \frac{1}{r}+ \frac{C'}{ \sqrt{\beta n} }+ {e^{-n\beta/2}}.\]
where $C'$ is an absolute constant. 
\end{lemma} 
 
The following proof is an adaptation of the proof of  Littlewood-Offord estimate \cite[Theorem 6.3]{NW21}. 
This theorem is proven  under different conditions, where instead of condition (2), a stronger assumption on the measure of all subgroups is imposed. 

The proof uses the Fourier transform of measure and some of 
Let $\nu$ be a probability measure on $\ZZ_N$. Write $e_N(x)= \exp( 2 \pi i x/ N)$ for $x \in \ZZ_N$. The following basic property of 
$e_N$ is used several times in the sequel. For each $a \in \ZZ_n$, we have $\sum_{x \in \ZZ_n} e_N(ax)$.  We will define the (Fourier) transform $ \widehat{\nu}: \ZZ_N \to \CC$ by 
 $$ \widehat{\nu}(t) :=   \sum_{ a \in \ZZ_N } \nu(a) e_N( at).$$

\begin{lemma}\label{fourier}
Let $\nu$ be a probability measure on $\ZZ_N$. Then we have
\begin{enumerate}
\item $ \widehat{\nu}(0)=1$. 
\item For all $n \ge 1$, we have $ \widehat{\nu^{(n)}} = (\widehat{\nu})^n$.
\item $\nu(0)= \frac{1}{N} \sum_{t \in \ZZ_N} \widehat{\nu}(t). $
\end{enumerate}
\end{lemma}

\begin{proof}
Part (1) is clear. For (2), suppose that $\nu_1$ and $\nu_2$ are two probability measures on $\ZZ_N$. Then we have 
\begin{equation}
\begin{split}      
\widehat{ \nu_1 \ast \nu_2} (t) &=  \sum_{ a \in \ZZ_N} (\nu_1 \ast \nu_2)(a) e_N(at) \\
&=   \sum_{ a \in \ZZ_N}  \sum_{ a_1+ a_2=a } \nu_1 (a_1) \nu_2(a_2) e_N(a_1t) e_N(a_2t) \\
&= \sum_{ a_1 \in \ZZ_N}  \nu_1 (a_1)  e_N(a_1t)  
   \sum_{ a_2 \in \ZZ_N}  \nu_2 (a_2)  e_N(a_2t)  = \widehat{ \nu_1} (t)
   \widehat{\nu_2} (t).
\end{split}
\end{equation}
Now, (2) follows by induction. For (3) note that
\begin{equation}
\begin{split}      
 \frac{1}{N} \sum_{t \in \ZZ_N} \widehat{\nu}(t) &=
   \frac{1}{N} \sum_{t \in \ZZ_N}  \sum_{ a \in \ZZ_N } \nu(a) e_N( at)   \\
  & =   \frac{1}{N} \sum_{a \in \ZZ_N}  \nu(a)  \sum_{ t \in \ZZ_N } e_N( a) = \nu(0),
\end{split}
\end{equation}
where the last equality follows from the fact that for every $a \in \ZZ_N$, we have $  \sum_{ t \in \ZZ_N } e_N( at)=0$  unless $a=0$, in which case it is equal to $N$. 
\end{proof}

We can now start with the proof of \Cref{lem:nonconcentration2}. 
\begin{proof}[Proof of \Cref{lem:nonconcentration2}]
Using parts (2) and (3) of  \Cref{fourier} we have
 \[ \nu^{(n)}(0)= \frac{1}{N} \sum_{  t \in \ZZ_N}  \widehat{\nu}(t)^n, \]
Applying riangle inequality we deduce
 \[    \nu^{(n)}(0)    \le \frac{1}{N} \sum_{ t  \in \ZZ_N } | \ \widehat{\nu} (t) |^n. \]
Writing $\psi(t)= 1- | \widehat{\nu}(t)|^2$ and using the inequality $ |x| \le \exp  \left( - \frac{1-x^2}{2 } \right) $ that holds for 
all $x \in \RR$ we obtain
\[  \nu^{(n)}(0)    \le   \frac{1}{N} \sum_{ t  \in \ZZ_N } \exp  \left( - \frac{n}{2} \psi(t) \right). \]
Set $f(t)= n \psi(t)$ and define $ T(w)=  \{ t: f(t) \le w \} $. Note that since $f(0)=n \psi(0)=0$, hence 
$0 \in T(w)$ for all $w \ge 0$. 

By separating the sum into level sets we obtain the following inequality:
\begin{equation}\label{bound1}
 \nu^{(n)}(0)      \le \frac{1}{N} \int_0^{ \infty}  |T(w)| e^{-w/2} dw. 
\end{equation} 

For an integer $k \ge 1$, and subsets $A_1, \dots, A_k \subseteq \ZZ_N$ we write
\[ A_1+ \cdots + A_k:= \{ a_1+ \cdots + a_k: a_i \in A_i, \quad 1 \le i \le k \}. \]
We use the shorthand $kA$ for $A+ \cdots +A$, where $k$ is the number of summands. 
The following lemma is proven in \cite[Proposition 3.5]{Maples} for all finite fields. A simple verification shows that it applies 
verbatim to all finite cyclic groups. For reader's convenience we will sketch the modification needed in the proof. 

\begin{lemma}\label{kTW}
For any $w>0$ and integer $k \ge 1$, we have 
\[ k T(w) \subseteq T( k^2 w). \]
\end{lemma}

\begin{proof}
It suffices to show that for all $ \beta_1, \dots, \beta_k \in \ZZ_N$ we have
\[ \psi( \beta_1+ \cdots + \beta_k) \le k ( \psi(\beta_1) + \cdots + \psi( \beta_k)  ). \]
This, in turn can be re-written as
\[ 1- \sum_{ a,b \in \ZZ_N} \mu(a) \mu(b) \cos ( \frac{2\pi}{N}(a-b)( \beta_1+ \cdots + \beta_k) ) 
\le k^2- k \sum_{j=1}^k \sum_{ a, b \in \ZZ_N} \mu(a) \mu(b) \cos ( \frac{2\pi}{N} (a-b) \beta_j). \]
This follows from the trigonometric inequality proven in \cite[Proposition 3.1]{Maples}. 
\end{proof}

We will also need a lemma from additive combinatorics. Define the set of symmetries of a set $A \subseteq \ZZ_N$ by
$\Sym A:= \{ h \in \ZZ_N: h+ A = A \}$.  
Note that $\Sym A$ is a subgroup of $\ZZ_N$ and if $0 \in A$ then $ \Sym (A) \subseteq A$. 
The lemma follows by a simple inductive argument from \cite[Theorem 5.5]{TaoVu}.

\begin{lemma}[Kneser's bound]\label{Kneser}
Let $Z$ be an abelian group and $A_1, \dots, A_k \subseteq Z$ are finite subsets. Then we have 
\[ |A_1+ \cdots + A_k| + (k-1) | \Sym (A_1+ \cdots + A_k)| \ge |A_1|+ \cdots + |A_k|. \]
In particular, for any finite $A$ we have 
\[ k|A| \le |kA|+ (k-1)  |\Sym(kA)|, \]
where $kA= A+ \cdots +A$, with the sum containing $k$ copies of $A$. 
\end{lemma}

We will now claim that if $H$ is a subgroup of $ \ZZ_N$ with $|H| \ge \frac{N}{r}$ then there exists $t \in H$ with 
$f(t) \ge n \beta$. To prove this note that 
\[ \frac{1}{|H|}   \sum_{ t \in H } f(t) = \frac{n}{|H|}  \sum_{ t \in H } ( 1-  | \widehat{\nu} ( t)|^2) . \]
Expanding the definition of $ \widehat{\nu}$ in $| \widehat{\nu}(t)|^2=  \widehat{\nu}(t) \cdot \overline{ \widehat{\nu}(t)}$, and summing for $t \in H$ we obtain
\begin{equation}
\begin{split}      
\sum_{ t \in H }  | \widehat{\nu} ( t)|^2 & =\sum_{ t \in H }   \sum_{ \theta_1 \in  \ZZ_n } \nu(\theta_1)  e_N(\theta_1 t)   \sum_{ \theta_2 \in \ZZ_n } \nu(\theta_2) e_N( -\theta_2t) \\ 
&= \sum_{ t \in H }   \sum_{ \theta_1, \theta_2 \in \ZZ_n } \nu(\theta_1)  \nu(\theta_2) e_N( ( \theta_1-\theta_2)t) 
=  |H| \cdot   \sum_{ \theta_1, \theta_2 \in \ZZ_N} \nu( \theta_1) \nu(\theta_2) 
\1_{ H^\perp} ( \theta_1- \theta_2) \\
& = |H|  \cdot  \sum_{ \theta_1 \in \ZZ_N} \nu( \theta_1)  \nu( \theta_1+ H^\perp).
\end{split}
\end{equation}
Here, $H^\perp$ denotes the dual of $H$, consisting of $x \in \ZZ_N$ such that $e_N( hx)=1$ for all $h \in H$. Note that $H^\perp$ is a subgroup of $\ZZ_N$.
If $\theta_1 \not\in H^\perp$ then $ 0 \not\in  \theta_1+ H^\perp$ and hence  $\nu( \theta_1+ H^\perp) \le 1- \nu (0) \le 
1- \beta$. If $\theta_1 \in H^\perp$ then $ \theta_1+ H^\perp=   H^\perp$. Since $|H^\perp| \le r$, we have 
$ \nu (\theta_1+ H^\perp)  \le 1- \beta$ in this case as well. This implies that 
$ \sum_{ t \in H }  | \widehat{\nu} ( t)|^2 \le 1- \beta$. This shows that 
$ \frac{1}{|H|}   \sum_{ t \in H } f(t)  \ge n \beta,$ implying that there exists $t \in H$ such that $f(t) \ge n \beta$. 

Let us now suppose $ w < n \beta$. Let $k$ be  the largest positive integer with $k^2 w< n \beta$. We claim that 
$| \Sym ( k T( w)) | \le \frac{N}{r}$. In fact, if this is not the case, using Lemma \ref{kTW} we have
\[ k T( w) \subseteq T(k^2 w), \]
it follows that $T(k^2 w)$ contains a subgroup $H$ with more than $N/r$ elements.  We will then have
\[  H \subseteq T( k^2 w) \subseteq T( n \beta) \]
which  contradicts that claim proved above. It follows from Lemma \ref{Kneser} that for all $w < n \beta$, we have 
\[ |T(w)| \le \frac{1}{k} |T(k^2 w) |+ \frac{N}{r}. \]
By the choice of $k \ge 1$ we have 
\[ (2k)^2 w \ge (k+1)^2 w \ge n \beta \]
implying that $k \ge \sqrt{ \frac{n\beta}{4w}}$. Putting all these together we obtain
\[ |T(w)| \le \sqrt{ \frac{4w}{n \beta}} N + \frac{N}{r}. \] 
Inserting the latter bound in the range $w< n \beta$ and the trivial bound $|T(w)| \le N$ for $w >n \beta$
into \eqref{bound1}, we arrive at 
\begin{equation}
\begin{split}      
 \nu^{(n)}(0)   & \le \frac{1}{2N} \int_0^{ n \beta} (\sqrt{ \frac{4w}{n \beta}} N + \frac{N}{r})  e^{-w/2} dw+ \frac{1}{2N} \int_{n\beta}^{ \infty} Ne^{-w/2} dw \\
& \le \frac{C}{ \sqrt{ n \beta} } \int_0^{ \infty} \sqrt{ w}  e^{-w/2} dw + \frac{1}{r}+ \frac{1}{2} \int_{n\beta}^{ \infty}  e^{- w/2} dw \\
& \le \frac{1}{r}+  \frac{C'}{ \sqrt{ n \beta} }+ e ^{-n\beta/2} 
\end{split}
\end{equation}
This ends the proof of Lemma \ref{lem:nonconcentration2}.

\end{proof}
 
\begin{lemma}\label{lem:nonconcentration3}
Let $\nu$ be a probability measure on $\ZZ_N$ and $0< \beta  < 1/2$ such that $\beta \le \nu(0) \le 1- \beta$. 
Then for some constant $C'$ and all $n \ge 1$ we have
\[ \nu^{(n)} (0)   \le  \frac{1}{2}  + \frac{C'}{  \sqrt{ \beta n} }+e^{-n \beta/2}.\]
\end{lemma}  
 
\begin{proof}
This is proven  is similar to the proof of  Lemma \ref{lem:nonconcentration2}. As before we write 
\[   \nu^{(n)}(0)     \le \frac{1}{N} \sum_{ t \in \ZZ_N } | \ \widehat{\nu} (t) |^n
  \le \frac{1}{2N} \int_0^{ \infty}  |T(w)| e^{-w/2} dw. \]
We now claim that  for all $ w< n \beta$ we have
  \[ |T(w)| \le \sqrt{ \frac{4w}{n \beta}} N + \frac{N}{2}. \] 
To show this, for $ w < n \beta$ denote by $k$  the largest positive integer with $k^2 w< n \beta$. We claim that 
$| \Sym ( T( w)+ \cdots + T(w)) | \le \frac{N}{2}$. Suppose that this is not the case. 
Recall that $ T(w)=  \{ t: f(t) \le w \} $, where $f(t)= n \left( 1- | \widehat{\nu} (t)|^2 \right)$. Since $ \widehat{\nu}(0)=1$ we have $f(0)=0$. Since $w>0$, we have $0 \in T(w)$. 
This implies that 
\[ \Sym (T( w)+ \cdots + T(w) ) \subseteq T( w)+ \cdots + T(w) \subseteq T(k^2 w), \]
it follows that $T(k^2 w)$ contains a subgroup with more than $N/2$ elements, hence has to be equal to
$\ZZ_N$, from which it follows that  $T( n \beta)= \ZZ_N$. On the other hand, using the Plancherel formula we have
 \[ \frac{1}{N}\sum_{ t \in \ZZ_N }  | \widehat{\nu} ( t)|^2 =  \sum_{ \theta_1 \in \ZZ_N} \nu( \theta_1)^2 
 \le (1-\beta)^2 + \beta^2  \le  1- \beta, \]
where the second inequality follows from $ (1- \beta)- (1-\beta)^2 - \beta^2 = \beta( 1-2 \beta)>0$.
 This implies that $ \frac{1}{N} \sum_{t \in \ZZ_N} (1- | \widehat{\nu}(t)|^2)  \ge \beta$. 
Hence there exists $t \in \ZZ_N$ such that $f(t) \ge n \beta$.  A similar argument as above finishes the proof. 
\end{proof}

We will now combine Lemmas \ref{lem:nonconcentration1}, \ref{lem:nonconcentration2}, \ref{lem:nonconcentration3} to prove a non-concentration theorem for random walks on  abelian groups with cyclic torsion component. 
Let $\Lambda:= \Lambda_1 \times \Lambda_2$ where $\Lambda_1 \sim \ZZ^s$ and $\Lambda_2 \simeq \ZZ_N$ are free 
and torsion components of $\Lambda$.  
For $i=1,2$, denote the projection from $\Lambda$ onto $\Lambda_i$ by 
$\pi_i$, and for probability measure $\nu$ on $\Lambda$, we write $\nu_i:=(\pi_i)_{\ast} \nu$ for the push-forward of 
$\nu$ under $\pi_i$.

\begin{theorem}[Non-concentration for random walks on abelian groups]\label{return}
Let $\Lambda$ be a  finitely generated abelian group with a cyclic torsion subgroup and $\nu$ a finitely supported probability measure on $\Lambda$. Let $ 0< \beta <1/2$ and $r \ge 1$ be a positive integer. Assume that 
\begin{enumerate}
\item $ \beta \le \nu(0) \le 1- \beta $.
\item $\nu(\Lambda^{{\tiny{\le r}}}) \le 1  - \beta$, where $\Lambda^{\le r}$ is the subset of 
$ \Lambda$ consisting of elements of order at most $r$.
\end{enumerate}
Then for all $n \ge 1$ we have 
\[  \nu^{(n)}(0)  \le  \frac{1}{r}+ \frac{C''}{ \sqrt{ n \beta} }+ e^{-n \beta/\textcolor{black}{4}}, 
\]
where $C''$ is an absolute constant. Moreover, only under assumption (1) above, we have 
 \[  \nu^{(n)}(0)  \le  \frac{1}{2}+ \frac{C''}{ \sqrt{ n \beta} }+ e^{-n \beta/\textcolor{black}{4}}. 
\]
\end{theorem}

\begin{proof}
 
Clearly we have
\[ \nu^{(n)}(0) \le \min ( \nu_1^{(n)}(0),  \nu_2^{(n)}(0) ). \]
  
Moreover, the inequality $\beta \le \nu_j(0)$ is satisfied for both $j=1,2$.
We consider two cases: if $  \nu_1(0) \le 1- \frac{\beta}{4}$, then it follows from Lemma \ref{lem:nonconcentration1} that
\[\nu^{(n)}(0) \le   \nu_1^{(n)}(0) \le \frac{2C}{ \sqrt{ \beta n } }. \]
Hence, suppose that  $  \nu_1(0) > 1- \frac{\beta}{4}$.
We claim that in this case we must have $   \nu_2(0) \le 1-  \frac{\beta}{2}$. If this is not the case, then 
$\nu_2(0) \ge 1 - \frac{1}{2} \beta$ which contradicts assumption (2). We will now verify condition 
(2) of Lemma \ref{lem:nonconcentration2} by finding an upper bound on the measure (under $\nu_2$) of the set of elements of order at most $r$ in $\Lambda_2$. 
\begin{equation}
\begin{split}      
\nu_2  \{ z \in \Lambda_2: \ord_{\Lambda_2}(z) \le r \} & = \nu \{ z= (z_1, z_2) \in \Lambda : \ord_{\Lambda_2} ( z_2) \le r \}  \\
& \le  \nu (  \{ z= (0, z_2) \in \Lambda : \ord ( z) \le r \} ) + 
\nu  ( \{ z= (z_1, z_2) \in \Lambda : z_1 \neq 0 \} ) \\
& \le 1- \beta+ \frac{\beta}{4} < 1- \frac{\beta}{2}. 
\end{split}
\end{equation}
In particular, if $H$ is a subgroup with $|H| \le r$, then the order of every element of $H$ is at most $r$, and hence
$\nu_2(H) \le 1 - \frac{\beta}{2}$. We can now apply Lemma \ref{lem:nonconcentration2} and Lemma \ref{lem:nonconcentration3} (with $\beta/2$ instead of $\beta$) to obtain the result. 
\end{proof}

\begin{corollary}\label{highpower}
Let $A$ be a matrix in $\GL_d(\CC)$ and $\beta>0$. Assume that
\begin{enumerate}
\item $\beta \le \M_1(A) \le 1-\beta$. 
\item $\M^{\le r}(A) \le (1- \beta)$
\end{enumerate}
Then for every $ \eta>0$ there exists $M:=M(\beta, r, \eta)$ such that 
for all $n \ge M$ we have $$\M_1(\T^n A) \le \frac{1}{r}+ \eta.$$
Moreover, only under assumption (1), for every $ \eta >0$ there exists $M':=M'(\beta, \eta)$ such that, 
for all $n \ge M'$ we have $$\M_1(\T^n A) \le \frac{1}{2}+ \eta.$$
\end{corollary}

\begin{proof}
Denote by $ \Lambda$ the subgroup of $\CC^{\ast}$ generated by the eigenvalues of $A$. Denote the subgroup of torsion elements in $\Lambda$ by $\Lambda_2$ and let $\Lambda_1 \simeq \ZZ^s$ be a subgroup of $ \Lambda$ such that $ \Lambda= \Lambda_1 \times \Lambda_2$. First suppose that conditions (1) and (2) hold. These guarantee that $\xi_A$ satisfies (1) and (2) of  \Cref{return}. 
It follows that for all $n \ge 1$ we have 
\[  \xi_A^{(n)}(0)  \le  \frac{1}{r}+ \frac{C''}{ \sqrt{ n \beta} }+ e^{-n \beta/{4}} . \]
Given $\eta>0$, we can find $M=M(\beta, \eta)$ such that for all $n \ge M$ we have 
\[ \frac{C''}{ \sqrt{ n \beta} }+ e^{-n \beta/{4}}< \eta. \]
This inequality together with the bound $ \M_1(\T^n A) = \xi_A^{(n)}(1)$ establishes the claim.
The second part of the statement (involving only assumption (1)) follows by a similar argument. 
\end{proof}

\section{Dynamics of the number of Jordan blocks under tensor powers}\label{sec:jordan}
In this section, we will study the effect of amplification on matrices $A$ with $\M_1(A)$ close to $1$. 
If $A$ is such a matrix, for $\rho(A-\Id)$ to be away from zero, a definite proportion of its Jordan blocks have to be of size at least $2$. We will show that under this condition, the proportion of the number of Jordan blocks to the size of matrix tends to zero. Moreover, we will give an effective bound for the number of iterations required to reduce this ratio below any given positive $\epsilon$. 

We will denote by 
$J(\alpha , s) $ a Jordan block of size $s$ with eigenvalue $ \alpha$. 
Let $A$ be a $d \times d$ complex matrix. We denote by $\J(A)$ the number of Jordan blocks in the Jordan decomposition of $A$ divided by $d$, hence $0 \le \J(A) \le 1$, with $ \J(A)=1$ iff $A$ is diagonalizable.   Define a sequence of matrices by setting 

\begin{equation}\label{iterate}
A_1 = A, \quad A_m = A_{m-1} \otimes A_{m-1}, \quad m \ge 1.
\end{equation}
The main result of this section is the following theorem:

\begin{theorem}\label{jdec}
For every $  \gamma>0$ and every $\eta>0$ there exists $N(\gamma, \eta)$ such that if 
$ \J(A) \le 1- \gamma$ and $m>N(\gamma, \eta)$ then $ \J(A_m)< \eta$, where $A_m$ is defined by 
\eqref{iterate}. 
\end{theorem}

The proof of this theorem will involve estimating certain trigonometric integrals. Before reformulating the problem, we will recall some elementary facts.

\begin{lemma}[\cite{AP17, MR2482411}, Lemma 5.7]\label{lemma:jordan}
For $m,n \in \mathbb{N}$ and $\alpha , \beta \in \CC$ we have:
$$J(\alpha, m) \otimes J(\beta, n) =  \bigoplus_{i = 1}^ {\min(m,n)} J(\alpha\beta, m+n + 1 - 2i).$$
\end{lemma}

\begin{remark}\label{simplefact}
It follows from this lemma that the number of Jordan blocks of a given size in tensor powers of $A$ only depends on 
the size of Jordan blocks in $A$ and not on the corresponding eigenvalues. We will use this simple fact later. 
\end{remark}


In order to study the asymptotic behavior of the number of Jordan blocks in tensor powers of $A$, it will be 
convenient to set up some algebraic framework. Let us denote by $\A$ the set of all $n \times n$ matrices in the Jordan normal form
with eigenvalue $1$ on the diagonal, where $n \ge 1$ varies over the set of all natural numbers. Note that if $A \in \A$ then $A \otimes A \in \A$. We will denote by $\E$ the set of all formal linear combinations $  \sum_{ n \ge 1 }  c_n \delta_n$, where $c_n \ge 0$ are integers
and $c_n=0$ for all but finitely many values of $n$. We equip $\E$ with a binary operation defined by 
$$\delta_m \ast \delta_n = \sum_{i=1}^{\min(m,n)} \delta_{m+n+1-2i}$$
and extended linearly to $\TP$.  Finally, for $n \ge 1$, consider the Laurent polynomial
$$T_n(x) = \sum_{i=1}^n x^{n - 2i + 1} = x^{n-1} + x^{n-3} + ... + x^{-(n -3)} + x^{-(n-1)},$$
and denote by $\TP$ the set of all (finite) integer linear combinations of $\{ T_n(x)\}_{n \ge 1}$. Note that since $T_n$ have different degrees, they do form a basis for the $\ZZ$-module $\TP$.

Now we will construct natural maps between these objects that will allow us to encode the number and size of Jordan blocks in tensor powers of a matrix using polynomials $T_n(x)$.  First, define the map $\Delta: \A \to \E$ as follows. Let $A \in \A$ be a matrix with $c_i$ blocks of size $i$ for $i \ge 1$. Then define 
\[ \Delta(A)=  \sum_{ n \ge 1} c_n \delta_n. \]
Note that $\Delta$ is a bijection. We now define $\Theta: \E \to \TP$ by 
\[ \Theta  \left( \sum_{ n \ge 1 }   c_n \delta_n \right) = \sum_{ n \ge 1}  c_n T_n(x). \]

\begin{lemma}\label{conv}
The following hold:
\begin{enumerate}
\item For all $A_1, A_2 \in \A$, we have $\Delta (A_1 \otimes A_2)= \Delta(A_1) \ast \Delta (A_2).$
\item For all $x_1, x_2 \in \E$ we have $ \Theta (x_1 \ast x_2)= \Theta(x_1) \Theta(x_2).$
\end{enumerate}
\end{lemma}
\begin{proof}
For (1), note that when $A_1$ consists of one block of size $m$ and $A_2$ is a block of size $n$, this is a restatement of \Cref{lemma:jordan}. In this case, we have $\Delta (A_1)= \delta_m$ and $\Delta(A_2)= \delta_n$ and $\Delta(A_1 \otimes A_2)$ is precisely  the expression defined by $\delta_m \ast \delta_n$. It follows immediately from the definition of tensor product of matrices that the 
equality extends to all linear combinations with non-negative integer coefficients. This establishes (1). 

In order to prove (2), we will show that For $m,n \in \mathbb{N}$, the following identity holds:
$$T_m(x) T_n(x) = \sum_{i=1}^{\min(m,n)} T_{m+n+1-2i} (x).$$
Without loss of generality, assume that $m \le n$. We will proceed by calculating the coefficient of $x^{r}$ for $ r \in \ZZ$ on both sides. Note that $x^r$ appears on either the right-hand side or the left-hand side if and only if $ r = m+n - 2j$ for some $1 \le j \le m + n - 1$. Since the coefficients of $x^j$ and $x^{-j}$ are equal, it suffices to consider the coefficient of $x^r$ for non-negative values of $r$, which correspond to  $0 \le j \le \frac{m+n}{2}$. Fix $j$ in this range. It follows from the definition that  $x^{m+n - 2j}$ appears in $T_{m+n + 1 - 2k}$ if and only if $m+n - 2k \ge m+n- 2j $, or, equivalently if $k \le j$. Since we have $ 1 \le k \le m$, the term $x^{m+n - 2j}$ appears exactly $\min( j , m)$ times on the right-hand side.
We now show that the coefficient of $x^{m+n  - 2j}$ on the left-hand side is the same. Note that the coefficient of the term 
$x^r$ with $r=m+n-2j$ in $T_m T_n$ is equal to the number of pairs $(j_1, j_2)$ such that $x^{m + 1 - 2j_1}$ appears in $T_m$ and $x^{n + 1 - 2j_2}$ appears in $T_n$ where $(m+1-2j_1)+(n+1-2j_2)= m+n - 2j$. These conditions can be reformulated as bellow:  
\begin{equation}\label{counting}
 j_1 + j_2 = j + 1 \qquad  1 \le j_1 \le m , 1 \le j_2 \le n.
\end{equation}
Let us distinguish two cases: First suppose that $j \le m$. In this case we choose $1 \le j_1 \le j$ and any such choice determines
$j_2= j+1-j_1$ uniquely. Hence the number of solutions to \eqref{counting} is $j$.  When $j>m$, then $j_1$ is allowed to vary in the range $1 \le j_1 \le m$, and since $j \le \frac{m +n}{2}$, $j_2:=j+1-j_1$ will automatically satisfy the inequality 
$1 \le j_2 \le n$. In conclusion, the coefficient of $x^{m+n-2j}$ on the left-hand side equals $\min(j,m)$ as well. This proves the claim. It thus follows that $ \Theta (x_1 \ast x_2)= \Theta(x_1) \Theta(x_2)$ holds for $x_1= \delta_m$ 
and $x_2= \delta_n$. Since $\Theta $ is extended to $ \E$ by linearity, the general case follows immediately. 
\end{proof}

\begin{corollary}\label{gen1}
Let $A$ be a $d \times d$ invertible complex matrix. Let $a_n^{(k)}$ be the multiplicity of Jordan blocks of size $n$ in the Jordan decomposition of $A_k$.
$$  \left( \sum_{n \ge 1}a_n^{(1)}T_n(x) \right) ^{2^{k-1}} = \sum_{n \ge 1} a_n^{(k)}T_n(x)$$
\end{corollary}
\begin{proof}
Using \Cref{simplefact}, we can assume that all eigenvalues are equal to $1$, or $A \in \A$. It is clear that 
$\Delta (A)= \sum_{n \ge 1} a_n^{(1)} \delta_n$. Part (1) of \Cref{conv}  implies that 
\[ \Delta (A_m)= \Delta(A_{m-1} \otimes A_{m-1})= \Delta(A_{m-1}) \ast \Delta(A_{m-1}). \]
Since $\Delta (A_m)= \sum_{n \ge 1} a_n^{(m)} \delta_n$, applying $\Theta$ to the previous equation and using part (2) of \Cref{conv} give
\[   \sum_{n \ge 1} a_n^{(m)}T_n(x)=  \left( \sum_{n \ge 1} a_n^{(m-1)}T_n(x) \right)^2. \]
 The claim follows by induction on $m$. 
\end{proof}

It will be more convenient to work with a trigonometric generating function. This is carried out in the next lemma.

\begin{corollary}\label{gen2}
With the same notation as in Corollary \ref{gen1}, set
$$P_k(\theta)= \frac{1}{d} \underset{n \ge 1}{\sum}  a_n^{(k)} \frac{\sin(n \theta)}{\sin(\theta)}.$$ 
For all $k \ge 1$ we have $$P_1(\theta)^{2^k} = P_k(\theta)$$
\end{corollary}

\begin{proof}
Substituting $x = e^{i\theta}$ into $T_n(x)$ yields
$$T_n(e^{i\theta}) = \sum_{i=1}^n e^{i(n - 2i + 1)\theta} 
= \frac{e^{in\theta} - e^{-in\theta}}{e^{i\theta} - e^{- i\theta} } = \frac{\sin(n\theta)}{\sin (\theta)}.$$
By applying  Corollary \ref{gen1} we obtain
$$ \left( \underset{n \ge 1}{\sum} a_n^{(1)} \frac{\sin(n \theta)}{\sin(\theta)} \right) ^{2^k} =\underset{n \ge 1}{\sum}  a_n^{(k)} \frac{\sin(n \theta)}{\sin(\theta)}.$$
 
\end{proof}

\begin{lemma}\label{sine}
Let $ 0 \le \epsilon \le 2$ and $2 \le n \in \mathbb{N}$. For any real number $\theta$ such that $|\cos(\theta)| \le 1 - \frac{\epsilon}{2}$ we have 
$$ \left|  \frac{\sin(n\theta)}{\sin(\theta)}  \right| \le n - \epsilon.$$
In particular, $|\frac{\sin(n\theta)}{\sin(\theta)}| \le n$ holds for all $\theta \in \mathbb{R}$. 
\end{lemma}

\begin{proof}
We will proceed by induction on $n$. For $n = 2$ the required inequality is obvious. Assume that it also holds for $n$. One can write:
$$ \left|   \frac{\sin((n+1) \theta)}{\sin (\theta)}  \right| = \left|  \frac{\sin(n\theta)\cos(\theta) + \sin(\theta)\cos(n\theta)}{\sin(\theta)}  \right| \le \left|  \frac{\sin(n\theta)}{\sin(\theta)}   \right| \cdot | \cos(\theta)| + |\cos(n\theta)| \le n - \epsilon + 1 $$
Thus the required inequality holds for $n+1$ as well and we are done.
\end{proof}

\begin{corollary}\label{p1}
For all $\theta \in \RR$ we have $|P_1(\theta)| \le 1$. 
\end{corollary}

\begin{proof}
This follows immediately from Lemma \ref{sine} and the equality $ \sum_{ n \ge 1} n a_n^{(1)} = d.$

\end{proof}

For a measurable subset $ B \subseteq \RR$ we denote by $\mu(B)$ is the Lebesgue measure of $B$.
We will also denote the number of Jordan blocks of size $1$ in $A$ by $ \J_1(A)$.

\begin{lemma}\label{bound}
Suppose that $A \in \GL_d(\mathbb{C})$ is in the Jordan normal form. For $ \epsilon>0$, set
$$B_{\epsilon}:= \{ \theta \in [ 0 , 2\pi] : |P_1(\theta)| \le 1 - (\J(A) - \J_1(A)) \epsilon \}.$$ 
Then
$$\mu([0 , 2\pi] \setminus B_{\epsilon}) \le  8 \sqrt{  \epsilon}.$$
\end{lemma}

\begin{proof} If all Jordan blocks are of size $1$, then $\J_1(A)= \J(A)$ and the claim follow Corollary \ref{p1}. 
More generally, suppose $\cos(\theta) \le 1 - \frac{\epsilon}{2}$. Then it follows from \Cref{sine} that
$$1 - |P_1(\theta)| =  \frac{1}{d}\sum_{m \ge 2}  \left(  m - \left|  \frac{\sin(m\theta)}{\sin(\theta)}  \right| \right) a_m^{(1)} \ge  \frac{1}{d} \sum_{m \ge 2} \epsilon a_m^{(1)}= \epsilon (\J(A) -\J_1(A))$$
This implies that $\theta \in B_{\epsilon}$. Hence
$$\mu([0 , 2\pi] \setminus B_{\epsilon}) \le 4 \cos^{-1}(1 -  \epsilon/2) \le 8 \sqrt{  \epsilon},$$
where one can easily verify the last inequality for all $ \epsilon \in (0,1)$. 
\end{proof}

\begin{lemma}
Suppose $A \in \GL_d(\CC)$ is not unipotent, that is, $\J(A) <1$. Then we have $\J_1(A \otimes A ) \le \J(A)^2$. 
\end{lemma}

\begin{proof}
First, notice that $J(\alpha,s) \bigotimes J(\beta,t)$ generates a Jordan block of size 1 if and only if the equation $s+t + 1 - 2i = 1$ holds for some $1 \le i \le \min(s,t)$. Equivalently $s+t = 2i$ for $1 \le i \le \min(s,t)$. This equation has a solution if and only if $s = t$. Therefore, we have  
$$\J_1(A \otimes A) =  \frac{\sum_{n \ge 1} (a_n^{(1)})^2}{d^2} \le   \left( \frac{ \sum_{n \ge 1}a_n^{(1)}}{d} \right)^2  = \J(A)^2.$$
\end{proof}

\begin{lemma}\label{trig}
For $a,b \in \mathbb{R}$, define
$I_n(a,b) =  \int_a^b \frac{\sin(n\theta)}{\sin(\theta)}d\theta$. Then for all  $n\in\mathbb{Z}$ we have
$$(n+1)\big(I_{n+2}(a,b) - I_n(a,b) \big) = 2 \big[ \sin((n+1)b) - \sin((n+1)a) \big].$$
\end{lemma}
\begin{proof}
As $a, b$ are fixed during this proof we write $I_n$ instead of $I_n(a,b)$. First note that
$$I_{n+2} = \int_a^b \frac{\sin((n+2)\theta)}{\sin(\theta)}d\theta = \int_a^b \frac{\sin(n\theta)\cos(2\theta)}{\sin(\theta)}d\theta+ \int_a^b \frac{\sin(2\theta)\cos(n\theta)} {\sin(\theta)}d\theta$$
Using $\cos(2\theta) = \cos^2(\theta) - \sin^2(\theta) = 1 - 2 \sin^2(\theta)$ we can write:
\begin{align*}
     I_{n+2} = &\int_a^b \frac{\sin(n\theta)(1 - 2 \sin^2(\theta))}{\sin(\theta)}d\theta+ 2\int_a^b \cos(\theta) \cos(n\theta) d\theta\\
      = & \int_a^b \frac{\sin(n\theta)}{\sin(\theta)}d\theta - 2 \int_a^b \sin(\theta)\sin(n\theta)d\theta + 2 \int_a^b \cos(\theta)\cos(n\theta) d\theta \\
      =& I_n +2 \int_a^b (\cos(n\theta)\cos(\theta) - \sin(n\theta)\sin(\theta) d\theta\\
      =& I_n +2 \int_a^b \cos((n+1)\theta)d\theta = I_n + \frac{2}{n+1}(\sin((n+1)b)-\sin((n+1)a).
\end{align*}
\end{proof} 

\begin{corollary}\label{bounds}
\begin{enumerate} 
 \item 
 For all odd integers $n$ we have $$I_n(0, 2 \pi) = 2\pi,  \quad I_n( \frac{\pi}{2} ,  \frac{3\pi}{2}) = \pi.$$
 \item 
 For all even integers $n$ we have $$ I_n(0, 2 \pi) = 0, \quad I_n( \frac{\pi}{2} ,  \frac{3\pi}{2}) \ge 3. $$
\end{enumerate}
\end{corollary}

\begin{proof}
A direct computation shows that $I_2(0, 2 \pi) = 0 $ and  $I_1 (0, 2 \pi)  = 2\pi $.  Repeated application of Lemma  \ref{trig} extends this to all integers $n$.  In a similar fashion, checking that $I_1( \pi/2, 3 \pi/2) = \pi $ together with 
 $\sin((n+1)\frac{\pi}{2})= \sin(3(n+1)\frac{\pi}{2}) = 0$ implies the claim for all odd values of $n$.  
The last  inequality is easy to check for $ n = 2 , 4$. 
For $ n \overset{4}{\equiv} 0 $ we have:
$$I_{n+2} = I_n - \frac{4}{n + 1} = I_{n-2} + \frac{4}{n-1} - \frac{4}{n + 1} > I_{n-2} \ge 3$$
and for  $n \overset{4}{\equiv} 2 $ we can write:
$$I_{n+2} = I_n + \frac{4}{n + 1} > I_{n-2} \ge 3.$$
\end{proof}

\begin{corollary}\label{boundbyintegrap}
Let $A \in \GL_d(\CC)$ and $P_k(\theta)$ be defined as above. Then we have 
\[ \J(A_k) \le \frac{1}{3} \int_0^{2 \pi} P_1(\theta)^{2^k} \ d\theta.  \]
\end{corollary}

\begin{proof}
Using \Cref{gen2} and \Cref{bounds} we can write:
\begin{align*}
  \int_0^{2 \pi}P_1(\theta)^{2^k} d\theta & \ge 
 \int_{\frac{\pi}{2}}^{\frac{3\pi}{2}}P_1(\theta)^{2^k} d\theta=  
   \sum_{n \ge 1} \frac{a_n^{(k)}}{d^{2^k}} \int_{\frac{\pi}{2}}^{\frac{3\pi}{2}} \frac{\sin(n\theta)}{\sin(\theta)}d\theta
    \\
    & \ge  3 \sum_{n\ge 1}\frac{a_{2n}^{(k)}}{d^{2^k}} + 2 \pi \sum_{n\ge 1}\frac{a_{2n-1}^{(k)}}{d^{2^k}}
    \ge  3 \J(A_k). 
\end{align*}
\end{proof}

\begin{proof}[Proof of Theorem \ref{jdec}]

Let $A \in \GL_d(\mathbb{C)}$ be such that $\J(A) \le \frac{1}{t}$, where $ t= (1-\gamma)^{-1}>1$.
Let $\tau < \frac{1}{2} (\gamma/(1-\gamma) )^2$. We will find a function $N_0(\gamma)$ such that 
for all $m \ge N_0(\gamma)$ we have $ \J(A_m) \le \frac{1}{t+\tau}$. By repeating this process
we see that for all $\ell \ge 1$ and all $m\ge \ell N_0(\gamma)$ we have 
  $ \J(A_m) \le \frac{1}{t+ \ell \tau}$. Let $k$ be defined by
$k=\left\lceil  ( \eta \tau)^{-1} \right\rceil +1$ so that $k\tau> \eta^{-1}$. It is clear that 
\[ N(\gamma, \eta)= k N_0(\gamma) \]
will satisfy the required property.

If we have 
$\J(A \otimes A) \le \frac{1}{t+\tau}$ we set $m=2$, and we are done. 
Note that Proposition 5.8. of \cite{AP17}, we have $\J(B \otimes B) \le \J(B)$ for any matrix $B$. Since $A_{m+1} = A_m \otimes A_m$, 
it follows that if we once the inequality holds for $m=2$ then the same inequality will continue to hold for all $m \ge 2$. Set let us assume that $ \J(A \otimes A) > \frac{1}{t+\tau}$. Then we have

\begin{align*}
    \int_0^{2\pi} P_1(\theta)^{2^k}d\theta =& \int_{B_{\epsilon}} |P_1(\theta)|^{2^k}d\theta +  \int_{[0,2\pi] \setminus B_{\epsilon}} |P_1(\theta)|^{2^k} d\theta \\
    \le &  \int_{B_{\epsilon}} (1 - \epsilon (\J(A) - \J_1(A)))^{2^k}d\theta +  \int_{[0,2\pi]  \setminus B_{\epsilon}}  d\theta \\
    \le & 2\pi  (1 - \epsilon (\J(A) - \J_1(A)))^{2^k}   + \mu([0,2\pi]\setminus B_{\epsilon}) \\
    \le & 2\pi   \left(  1 - \frac{1}{t+\tau}\epsilon + \frac{1}{t^2}\epsilon \right) ^{2^k} + 8 \sqrt{  \epsilon} 
\end{align*}
It follows from the choice of $\tau$ that $ \frac{1}{t+ \tau}> \frac{2}{t^2}$. This implies that 
\[   \int_0^{2\pi} P_1(\theta)^{2^k}d\theta  \le  2\pi   \left(1 - \frac{ \epsilon}{t^2} \right)^{2^k} + 8 \sqrt{  \epsilon}. \]
Set  $ \epsilon= \frac{1}{2^{8}(t+\tau)^2}< \frac{1}{2t^2}$ so that the second term is bounded by $ \frac{1}{2(t+\tau)}$. Now, let $N_0(\gamma)$ be the smallest positive integer $m$ such that 
\[ 2\pi   \left(1 - \frac{ \epsilon}{t^2} \right)^{2^m} <  \frac{1}{2(t+\tau)}. \]
It is clear that for this value of $m$ we have 
\[  \int_0^{2\pi} P_1(\theta)^{2^m}d\theta \le \frac{1}{t+\tau}. \]
The claim will now follow from Corollary \ref{boundbyintegrap}.
\end{proof}

\section{Proof of Theorem \ref{main-theorem}}\label{proof}
In this section, we will prove parts (1) and (2) of Theorem \ref{main-theorem}. The proof crucially depends on the inequality
\[ \rho(A, \Id) \ge \max( 1- \M_1(A), 1- \J(A)). \]
from Lemma \ref{basiclemma}.
Let $G$ be a torsion-free group and $S$ a finite subset of $G$ with $e \not\in S$. We will show that for 
every $ \epsilon>0$ and $\delta>0$ an $(S, \delta, 1- \epsilon)$-map into $\GL_D(\CC)$ exists. Let $r= \lfloor   2/ \epsilon \rfloor +1$ and 
set
$ \overline{S} = \{ g^{n}: g \in S, 1 \le n \le r! \}$. By Lemma \ref{addone}, there exists an $(\overline{S}, \delta_0, 0.23)$-map  $\phi_0$ such that 
\begin{equation}\label{lbound}
\M_1(\phi_0(g)) \ge 0.01 
\end{equation}
for all $g \in \overline{S}$. 
Here, $\delta_0$ is a small quantity whose values will be determined later.
We will find  $N$ such that for $n >N$, the tensor power $\T^{2^n} \phi_0$ is an $(S, \delta, 1- \epsilon)$-map. Fix some $g \in S$. We will consider two different cases. 

\vspace{1mm}

\noindent
{\it Case (1):} Suppose that $ \J(\phi_0(g^{r!}))< 0.99$. Then it follows from Lemma \ref{bad} and 
\Cref{jdec} that for $m = N(0.01, \epsilon/2)$ we have $\J( \T^{2^m}\phi_0(g^{r!})) \le \epsilon$. This implies that 
\[ \rho( \T^{2^m}\phi_0(g^{r!}), \Id) \ge 1- \epsilon/2 \]
On the other hand, using \Cref{bad}, we have
$$ \rho( \T^{2^m}\phi_0(g^{r!}),  \T^{2^m} \phi(g)^{r!} ) \le 2^{m} \rho( \phi_0(g^{r!}),   \phi(g)^{r!} ) \le 2^{m}r! \delta_0 < \epsilon/2 
 $$
as soon as $\delta_0 < \epsilon/ 2^{m+1}r!$. 
From here, we have by the triangle inequality that  
\[ \rho( \T^{2^m}\phi_0(g)^{r!}, \Id) \ge 1- \epsilon.\]
Note that if $B$ is any matrix and $m\ge 1$ then $\ker (B-I)  \subseteq \ker (B^m-I)$. In other words, we have
$\rho(B, I) \ge \rho(B^m, I)$. This implies that $\rho( \T^{2^m}\phi_0(g), \Id) \ge 1- \epsilon.$

\noindent
{\it Case (2):} Suppose that  $\J(\phi_0(g^{r!})) \ge 0.99$. The proportion of blocks corresponding to eigenvalue $1$ is $ \J_1(\phi_0(g^{r!}))$, and the proportion of other blocks is at most
$ 1- \xi_{ \phi_0(g^{r!})}(1). $ This implies that 
\[  \J_1(\phi_0(g^{r!})) +  1- \xi_{ \phi_0(g^{r!})}(1) \ge 0.99. \]
Since $ \J_1(\phi_0(g^{r!}) ) = 1- \rho( \phi_0(g^{r!}), \Id) \le 1-0.23= 0.77 $ we obtain
\[  \rho( \phi_0(g^{r!}), \Id) \ge 1- \M_1( \phi_0(g^{r!})= 1- \xi_{ \phi_0(g^{r!})}(1) \ge 0.22. \]
Also note that
\[ \rho( \phi_0(g)^{r!} , \phi_0(g^{r!})) \le   r! \delta_0. \]
It thus follows that
\[ \rho( \phi_0(g)^{r!} , \Id) \ge 0.22-r! \delta_0 \ge 0.21 \]
as long as $\delta_0<  \frac{1}{100 \, r!}$. If $D$ denotes the size of the matrix $\phi_0(g)^{r!}$, then we know that  $\phi_0(g)^{r!}$
has at least $0.99D$ Jordan blocks. If $k$ denotes the number of blocks corresponding to the eigenvalue $1$, then noting that each such block contributes a one-dimensional subspace to $\ker (\phi_0(g)^{r!}-I)$, we deduce that $k \le 0.21D$. 
Note that by the assumption  $\J(\phi_0(g^{r!})) \ge 0.99$, that is, $\phi_0(g^{r!})$ has at least $ 0.99D$ Jordan blocks. This means that the total number of eigenvalues (with multiplicity) coming from blocks of size at least $2$ is at most 
$0.01D$.  Hence, the multiplicity of $1$ as an eigenvalue of $ \phi_0(g^{r!})$ is at most $ 0.23D$.

Note that if $ ( \lambda_i)_{1 \le i \le d}$ are eigenvalues of 
$\rho(g)$ then $ ( \lambda_i^{r!})$  are eigenvalues of  $\rho(g)^{r!}$. If $ \lambda_i^k=1$ for some $ 1 \le k \le r$ then 
$  \lambda_i^{r!}=1$. This implies that 
\begin{equation}\label{ubound}
\M^{\le r}(\phi_0(g)) \le \M_1(\phi_0(g)^{r!}) = \xi_{\phi_0(g)^{r!}}(1)  \le 1- \beta.  
\end{equation}
for $\beta= 0.23$.  Applying Lemma \ref{highpower} it follows that for $m> M(\beta, r, 1/r)$, we have 
\[ \M_1(\T^m \phi_0(g)) \le \frac{2}{r} \le \epsilon. \]
It follows that $ \rho( \T^m \phi_0(g)), \Id ) \ge 1- \epsilon$. 
In conclusion, for $m> \max( M(\beta, r, 1/r),  N(0.01, \epsilon))$, we have 
\[   \rho( \T^m \phi_0(g)), \Id ) \ge 1- \epsilon. \]
for all $g \in S$. Note, however, that 
\[  \rho( \T^m \phi_0(g)  \T^m \phi_0(h),  \T^m \phi_0(gh) ) \le 2^{m} \delta_0. \]
Hence we need to choose $\delta_0< \min ( \frac{1}{100 \, r!}, \frac{\delta}{2^N}, \frac{\epsilon}{2^{m+1}r!})$ with 
$N=  \max( M(\beta, r, 1/r),  N(0.01, \epsilon))$ for the desired inequality to hold. 
The proof of (2) is similar and somewhat simpler. This time, we work directly with $S$ (and not $ \overline{S}$)
and apply part (2) of \Cref{highpower}.

\begin{remark}\label{why}
One can deduce from Theorem \ref{main-theorem} the more general version in which the field of complex numbers is replaced by an arbitrary field of characteristic zero $F$. In order to see this, suppose that $G$ is $\kappa$-linear sofic over $F$. This implies that 
for every finite set $S \subseteq G$ and every $ \delta>0$ and every $0 \le  \kappa' < \kappa$, there exists
$d \ge 1$ and a map $\phi: S \to \GL_d(F)$ that satisfies properties $(AH)$ and $(D)$ of Definition \ref{linsof}. Since $S$ is finite, one can  replace $F$ by a finitely generated subfield $F'$ of $F$ (depending on $S$). However, every such field is isomorphic to a subfield of $\CC$. This implies that the arguments given above show that one can amplify $\phi$. Now, note that all the amplifications are 
constructed via the functorial operations describe in \ref{functor}. This implies that the image remains in $\GL_m(F)$ for some $m \ge 1$, from which the claim follows. 
\end{remark}

\begin{remark}\label{remarkbad}
The part of the proof that is based on Lemma \ref{lemma:jordan} does not work over fields of positive characteristic. In fact, the entire section \ref{sec:jordan} uses heavily the special form of this formula. It would be interesting to see if the method can be generalized to fields of positive characteristic. 
\end{remark}

\section{Stability and the proof of Theorems \ref{main-theorem2} and \ref{main-theorem3}}\label{sec:finite}

In this section, we will prove Theorem \ref{main-theorem2}. Along the way, we will also address 
the more general question of determining $\kappa(G)$ when $G$ is an arbitrary finite group. As a byproduct, we will show that the bound $\kappa(G) \ge 1/2$ cannot be improved for finite groups. This will be carried out through computation of $\kappa( \ZZ_p^n)$, from which it will follow that as $n \to \infty$
$$\kappa(\ZZ_p^n) \to 1- \frac{1}{p}.$$ 
The special case of $p=2$ will then prove the claim.  One ingredient of the proof is the notion of \textit{stability} for linear sofic representations. 
Studying stability for different modes of metric approximation has been an active area of research in the last decade. 
Stability of finite groups for sofic approximation was proved by Glebsky and Rivera \cite{GR09}. Arzhantseva and
P\u{a}unescu \cite{AP-15} showed that abelian groups are stable for sofic approximation. This result was generalized by Becker, Lubotzky, and Thom \cite{BLT} who established a criterion in terms of invariant random subgroups for (sofic) stability in the class of amenable groups. For other related results, see for instance \cite{AP-15, CLT, BLT, BL} and references therein. Some progress towards proving the stability of $\ZZ^2$ in linear sofic approximation has been made in \cite{EG21}.   Our first theorem establishes the stability of finite groups in the normalized rank metric. Before stating and proving this result, we will need a simple fact from linear algebra.

\begin{lemma} Suppose that $W_1, \dots, W_i$ are subspaces of $\CC^d$ with $\dim (W_r) \ge d( 1-  \epsilon)$, for
$ 1 \le i \le r$.  Then, we have $\dim (\cap _{i=1}^r W_{i} )\ge d(1- r \epsilon)$.
\end{lemma}

\begin{proof}
We will proceed by induction on $r$. 
For $r=1$ there is nothing to prove. 
Assume that the claim is shown for $r-1$. Then, using the induction hypothesis, one can write:
\begin{align*}
    \dim(\cap_{i=1}^r W_{i} )   &=  \dim ( \cap_{i=1}^{r-1} W_{i} ) + \dim(W_{r}) - \dim (\cap_{i=1}^{k-1} W_{i} + W_{r}) \\
     & \ge  d ( 1 - (r-1) \epsilon) + d ( 1 - \epsilon) - d \\
     & =  d(1 - r \epsilon)
\end{align*}
\end{proof}

\begin{proposition}\label{stable}
Let $G$ be a finite group, $ \epsilon>0$ and $\varphi: G \to \GL_d(\mathbb{C})$  is such that for all  $g,h \in G$ we have
$$\rho( \varphi(g)\varphi(h) - \varphi(gh)) < \epsilon.$$
Then there exists a representation $\psi : G \to \GL_d(\mathbb{C})$ such that for every $g \in G$ we have
$$\rho(\varphi(g) , \psi(g) ) < |G|^2 \epsilon.$$
\end{proposition}

\begin{proof}
For $g, h \in G$ consider the subspace defined by 
$$W_{g,h} = \ker ( \varphi(g)\varphi(h) - \varphi(gh))$$ 
and set  $W = \cap_{g,h \in G} W_{g,h}$. We claim that $W$ is a $G$-invariant subspace of $\CC^d$. 
Assume that $w_1$ is an arbitrary element of $W$. It suffices to prove that for any $k \in G$, $\varphi(k)w_1$ is an element of $W$ as well. Since $w_1 \in W$, we can write:
\begin{align*}
    \varphi(g)\varphi(h)\big(\varphi(k)w_1\big) = & \varphi(g) \big( \varphi(h)\varphi(k)w_1\big) 
    \\
    = &\varphi(g)\varphi(hk)w_1 = \varphi(ghk)w_1 
    \\
    = & \varphi(gh)(\varphi(k) w_1)
 \end{align*}
This implies that $ \varphi(k)w_1 \in W_{g,h}$, proving the claim. In summary, $W \le \CC^d$ is a $G$-invariant subspace with the property that the restriction of $\psi(g)$ to $W$ is 
a representation of $G$. 
Let $W^{\perp}$ be a subspace complement of $W$. For $g \in G$, define $\psi(g) \in \GL_d(\CC)$ to be the linear transformation that acts on $W$ via $\varphi$ and on $W^{\perp}$ by identity. It is clear that $\psi$ defined in this way is a $G$-representations. 
Finally, for every $g \in G$ we have:
 \[     \rank(\psi(g) - \varphi(g)) \le \dim (W^{\perp} ) = d - \dim (W) \le d |G|^2 \epsilon.
\]
This finishes the proof. 
\end{proof}

\begin{remark}
It is noteworthy that our proof establishes stability with a linear estimate. However, the constant depends on $|G|$. It would be interesting to see if this dependency can be relaxed for certain families of groups. 
\end{remark}

We will start by providing a simple description of irreducible representations
of the group $G=\ZZ_p^n$. We start by setting some notation. Recall that $\e_p: \ZZ_p \to \CC^\ast$ denotes the character $\e_p(x)= \exp(2\pi i x/p)$. Also, for $x= (x_1, \dots, x_n ), y= ( y_1, \dots, y_n) \in \ZZ_p^n$, we write  $x \cdot y = \sum_{ i =1}^{n} x_i y_i$. For each $a \in \ZZ_p^n$, define
$\phi_a: \ZZ_p^n \to \CC^\ast$ by $\phi_a(x)= \e_p( a \cdot x)$. It is a well known \cite{Luong} fact that $\phi_a$, as $a \in \ZZ_p^n$ constitute all irreducible representations of $\ZZ_p^n$.

\begin{proposition}
For $n\ge 2$ we have $$\kappa(\ZZ_p^n) = \frac{p^n- p^{n-1}}{p^n-1}. $$
\end{proposition}

\begin{proof}
Let $$\phi:= \bigoplus_{ a \neq 0} \phi_a$$ denote the direct sum of all $\phi_a$ other than the trivial representation $\phi_0$. One can regard $\phi(x)$ as a diagonal matrix of size $p^n-1$ with diagonal entries $\phi_a(x)$ for non-zero $a \in \ZZ_p^n$. We claim that for every $x \neq 0$, the set of $\{ a \in  \ZZ_p^n \setminus \{ 0 \}: \phi_a(x)=1 \}$ has cardinality $p^{n-1}$. In fact, the condition $\phi_a(x)=1$ corresponds to the equation $ \sum_{ i =1}^{n } a_i x_i=0$ for $a$ which has exactly $p^{n-1}-1$ non-zero solutions. Hence $ \rk \phi(x)= p^n-p^{n-1}$, from which it follows that 
$$\rho(\Id ,  \phi(x) ) =   \frac{p^n- p^{n-1}}{p^n-1}.$$
To prove the reverse inequality, we first suppose $\psi$ is a $d$-dimensional representation of $G$. Decompose $\psi$ into a direct sum of irreducible representations $ \phi_a$ and denote the multiplicity of $\phi_a$ by $c_a$: 
$$ \phi= \bigoplus_a c_a \phi_a.$$
Set $$\beta:= \min_{g \in G} \frac{1}{d}\rk( \psi(g)- \Id), $$
and write $B(x)= \{ a \in \ZZ_p^n:  \phi_a(x) \neq 1\}$. 
This implies that for every non-zero $x\in \ZZ_p^n$, we have 
\begin{equation}\label{bin}
 \sum_{a \in B(x)} c_a \ge \beta d \hspace{2mm} \Rightarrow \hspace{3mm}  \sum_{ x \neq 0 } \sum_{a \in B(x)} c_a \ge \beta d  (p^n-1). 
\end{equation}
Without loss of generality, we can assume that $c_0=0$, since by removing the trivial representation the value of $\beta$ can only increase. 
Note also that for every $a \neq 0$, there are exactly $p^{n} - p^{n-1} $ elements $x \in \ZZ_p^n$ with $a \in B(x)$. This implies that  
\[  \sum_{ x \neq 0 } \sum_{a \in B(x)} c_a = (p^n- p^{n-1} ) \sum_{ a \neq 0 } c_a = d  (p^n- p^{n-1} ),\]
This together with \eqref{bin} implies that $ \beta \le \frac{p^n- p^{n-1}}{p^n-1}$.

Now, to finish the proof assume that $\beta:= \kappa(\ZZ_p^n)> \frac{p^n- p^{n-1}}{p^n-1}.$ Choose $ \epsilon>0$ such that $ p^{2n} \epsilon<  \beta -   \frac{p^n- p^{n-1}}{p^n-1}$. Let $ \varphi: \ZZ_p^n \to \GL_d(\CC)$ be such that 
$$\rho( \varphi(x)\varphi(y) - \varphi(x+y)) < \epsilon.$$
holds for all $x, y \in \ZZ_p^n$. Use Lemma \ref{stable} to find a representation $\psi: \ZZ_p^n \to \GL_d(\CC)$ such that 
$ \rho( \varphi(x), \psi(x) ) < p^{2n} \epsilon$. Since $ \rho( \varphi(x), \Id) \ge \beta$ for all non-zero $x \in \ZZ_p^n$, it follows that for all $x \neq 0$
\[ \rho ( \psi(x), \Id) \ge \beta - p^{2n} \epsilon  > \frac{p^n- p^{n-1}}{p^n-1}.\]
This is a contradiction. 
\end{proof}

\begin{corollary}
The best constant for the class of all groups is $1/2$. 
\end{corollary}

\subsection{The value of $\kappa(G)$ for finite groups}
Let $G$ be an arbitrary finite group. Note that it follows from 
\Cref{stable} that 
\[ \kappa(G)= \sup_{\psi} \min_{g \in G} \rho( \psi(g), \Id), \]
where $\psi$ ranges over all finite-dimensional representations of $G$. The first observation is that the supremum can be 
upgraded to a maximum.

\begin{proposition}\label{attained}
Let $G$ be a finite group. Then there exists a finite-dimensional representation 
$\psi: G \to \GL_d(\CC)$ of $G$ such that 
\[ \kappa(G)=  \min_{g \in G} \rho( \psi(g), \Id). \]
In particular, $\kappa(G)$ is a rational number. 
\end{proposition} 

\begin{proof}
Denote by $R= \{ \psi_0, \dots, \psi_{c-1} \}$ the set of all irreducible representations of $G$ up to isomorphism. Pick
$C= \{ g_0, \dots, g_{c-1} \}$ to be  a set of representatives for all conjugacy classes of $G$. We will assume that 
$\psi_0$ is the trivial representation and $g_0:=e$ is the identity element of $G$. Consider the $(c-1) \times (c-1)$ matrix
$K$ where $$K_{ij}= \dim \ker ( \psi_j(g_i)- \Id). $$ 
Note that $K_{ij}$ does not depend on the choice of the representative
$g_i$. Set $ \beta= \kappa(G)$, and write
\[ \Delta = \{ (x_1, \dots, x_{c-1}): x_i \ge 0, \sum_{ 1 \le i \le c-1 } x_i =1\}. \]
For each representation $ \psi: G \to \GL_d(\CC)$ which does not contain the trivial representation,
 we can decompose $\rho$ as $ \rho= \oplus_{j=1}^{c-1} n_i \psi_i$, and 
define  $\delta( \psi) \in \Delta $ to be the column vector 
\[ \delta( \psi) =   \left(  \frac{n_1}{d}, \dots, \frac{n_{c-1}}{d} \right)^t,  \]
of normalized multiplicities of irreducible representations of $G$. 
It is easy to see that $ \min_{g \in G} \rho( \psi(g), \Id) \ge \alpha$ holds iff for all $g \in G$
\[  \sum_{ i =1}^{c-1} n_i \ker (\psi_i(g)- \Id) \le (1- \alpha) d. \]
These conditions can be more succinctly expressed as
$$K \delta(\psi) \le (1- \alpha) (1,1, \dots, 1)^t,$$
where we write $x \le y$ for two vectors $x, y$ if every entry of $y-x$ is non-negative. 
By assumption, for every $ m \ge 1$, there exists a representation $ \psi_m $ such that $ K \delta(\psi_m) \le (1- \beta + \frac{1}{m}) (1,1, \dots, 1)^t$ holds. In other words, the set 
\[ \{ x \in \Delta:  K x \le (1- \beta + \frac{1}{m}) (1,1, \dots, 1)^t \} \]
is non-empty. By compactness of $ \Delta$, it follows that there exists a point $x \in \Delta$ such that 
$K x \le (1- \beta) (1,1, \dots, 1)^t  $. Note that these constitute a system of inequalities involving $x_1, \dots, x_{c-1}$ and $\beta$ with rational coefficients. Now using the fact the the set of solutions to this system is a rational polytope (or equivalently using the proof of Farkas’ lemma \cite{Mat}) we deduce that both $\beta $ and $(x_1, \dots, x_{c-1})$ are rational. This proves the claim. 

\end{proof}

\begin{remark}
 There are noncyclic finite groups $F$ for which $\kappa(F)=1$. For instance, let $F$ be a finite subgroup of $\SU_2(\CC)$. Such groups are cyclic of odd order and double covers
of finite subgroups of $\SO_3(\RR)$, which include the alternating groups of $4,5$ letters, and the symmetric group on $4$ letters. We claim that the natural representation $\rho $ of $F$ in $\GL_2(\CC)$
has the property that for $ g \neq e$ the eigenvalues of $\rho(g)$  
are not $1$. This is clear since if one eigenvalue is $1$, then the other has to be $1$ as well, contradicting the faithfulness of $ \rho$. These groups include, for instance,  $\SL_2(\FF_5)$.
 
\end{remark}

\begin{proposition}
Let $G$ be a finite group. Then $\kappa(G)=1$ iff $G$ has a fixed-point free complex representation. 
\end{proposition}

These groups have been classified by Joseph A. Wolf. The classification is rather complicated. We refer the reader to \cite{Wolf} for proofs and to \cite[Theorem (1.7)]{Nakaoka} for a concise statement and the table listing these groups. The difficulty of classifying finite groups $G$ with $\kappa(G)=1$ suggests that the problem of determining $\kappa(G)$ in terms of $G$ 
may be a challenging one.

\begin{remark}
Given a field $F$, one can also study the notion of linear sofic approximation over $F$. It is not known whether linear sofic groups
over $\CC$ and other fields coincide. However, one can show that
$\kappa_{\CC}(G)$ and $\kappa_F(G)$ do not need to coincide. This will be seen in the next subsection.
\end{remark}

\subsection{Optimal linear sofic approximation over fields of positive characteristic}\label{charp}
In this subsection, we will prove Theorem \ref{main-theorem3}

\begin{proof}[Proof of Theorem \ref{main-theorem3}]
Let $F$ be a field of characteristic $p$. It is easy to see that the proof of Proposition \ref{stable} works over $F$ without any changes. Hence, for any finite group $G$ we have
\[ \kappa_{F}(G)= \sup_{\psi} \min_{g \in G} \rho( \psi(g), \Id), \]
where $\psi$ runs over all representations $\psi: G \to \GL_d(F).$

Suppose $\rho:G \to \GL_d(F)$
is a linear representation. Let $g$ be an element of order $p$ in $G$.  Write $\psi(g)= \Id+ \tau(g)$ where $\tau(g)$ is a
$d \times d$ matrix over $F$. From $  \psi(g)^p=\Id$, and $F$ has characteristic $p$, it follows that 
$\tau(g)^p=0$. Let $J$ denote the Jordan canonical form of $\tau(g)$, consisting of $k$ blocks. 
It follows from $\tau(g)^p=0$ that each block in $J$ is a nilpotent block of size at most $p$, implying $kp \ge d$. Since $k=\dim \ker \tau(g)$ we have 
\[ \rk( \psi(g)- \Id) =d - \dim \ker \tau(g) \le d- \frac{d}{p}.  \]
This shows that $\kappa(G) \le 1- \frac{1}{p}$. 
To prove the reverse inequality, let $V$ denote the vector space consisting of all functions $f: G \to F$. Clearly $d:= \dim V= |G|$. Consider the left regular representation of $G$ on $V$ defined by 
\[ (\psi(g) f )(h)= f ( gh). \]

Let $g \in G \setminus \{ e \}$, and consider the subspace
\[ W(g)= \{ f \in V: \psi(g) f =f \}. \]
Any $f \in W(g)$ is invariant from the left by the subgroup $ \langle g \rangle$ generated by $g$. It follows that 
\[ \dim W(g) \le \frac{d}{| \langle g \rangle|} \le \frac{d}{p}.\]
Hence 
\[  \rk( \psi(g)- \Id) = d -  \dim W(g) \ge d- \frac{d}{p}, \]
proving the claim. 
 \end{proof}
 
\bibliographystyle{amsalpha}
\bibliography{bib-ran-2}

\end{document}